\documentclass[12pt]{amsart}

\newtheorem{theorem}{Theorem}[section]

\newtheorem{corollary}[theorem]{Corollary}

\newtheorem{lemma}[theorem]{Lemma}
\newtheorem{proposition}[theorem]{Proposition}

\theoremstyle{definition}

\newtheorem{definition}[theorem]{Definition}

\newtheorem{remark}[theorem]{Remark}

\newtheorem{example}[theorem]{Example}

\theoremstyle{parrafo}

\begin{document}

\title[]{Generalized Bernstein operators on the classical polynomial
spaces}

\author{J. M. Aldaz and H. Render}
\address{Instituto de Ciencias Matem\'aticas (CSIC-UAM-UC3M-UCM) and Departamento de 
	Matem\'aticas,
	Universidad  Aut\'onoma de Madrid, Cantoblanco 28049, Madrid, Spain.}
\email{jesus.munarriz@uam.es}
\email{jesus.munarriz@icmat.es}
\address{H. Render: School of Mathematical Sciences, University College
	Dublin, Dublin 4, Ireland.}
\email{hermann.render@ucd.ie}

\thanks{2010 Mathematics Subject Classification: \emph{Primary: 41A10}}
\thanks{Key words and phrases: \emph{Bernstein polynomial, Bernstein operator.}}

\thanks{The first named author was partially supported by Grant MTM2015-65792-P of the
	MINECO of Spain, and also by by ICMAT Severo Ochoa project SEV-2015-0554 (MINECO)}

\begin{abstract} 
We study generalizations of the classical Bernstein operators
on the polynomial spaces  $\mathbb{P}_{n}[a,b]$, where
instead of fixing $\mathbf{1}$ and $x$, we reproduce exactly $\mathbf{1}$ and 
a polynomial $f_1$,  strictly increasing on $[a,b]$. 
 We prove
that for sufficiently large $n$,
there always exist generalized Bernstein operators fixing  $\mathbf{1}$ and $f_1$.
These operators are defined by non-decreasing sequences of nodes precisely when 
$f_1^\prime > 0$ on $(a,b)$, but even
if $f_1^\prime$ vanishes somewhere inside $(a,b)$, they converge to the identity.
\end{abstract}

\maketitle

\markboth{J. M. Aldaz, H. Render}{Generalized Bernstein Operators}

\section{Introduction} 

Let $\mathbb{P}_{n}[a,b]$ denote the space of polynomials of degree
bounded by $n$, over the interval $[a,b]$.  
The classical Bernstein operator 
$B_{n}:C
\left[ a,b\right] \rightarrow \mathbb{P}_{n}[a,b]$, defined by 
\begin{equation}
B_{n}f\left( x\right)
 =
 \sum_{k=0}^{n}f\left( a+\frac{k}{n}\left( b-a\right)
\right) \binom{n}{k}\frac{\left( x-a\right) ^{k}\left(
	b-x\right) ^{n-k}}{\left( b-a\right) ^{n}},  \label{defBPr}
\end{equation}
reproduces exactly (or fixes) the affine functions, but from the design viewpoint, one might  be interested
in the precise reproduction of other functions.
So a natural idea 
is to search for analogous operators 
that fix all functions in a given two-dimensional space, possibly different from
the affine functions, and still converge to the identity. 

Here we explore such  generalized Bernstein operators $B_{n}^{f_1}$ on  polynomial spaces,
where fixing the constant function $\mathbf{1}$ and an injective  polynomial $f_1$ is achieved,
when possible, by modifying the location of the nodes $t_{n, k }$ 
(instead of having  $t_{n, k } =  a+\frac{k}{n}\left( b-a\right)$, as in (\ref{defBPr})). A motivation for this approach
is that it allows us to keep the Bernstein bases unchanged, a desirable feature given
their  several  optimality properties, cf. for instance \cite{Fa}. Multiplying by $-1$ if needed,
we may assume that $f_1$ is increasing.

We shall  show  (cf. Theorem \ref{Thm3})  that given
any  polynomial $f_1(x)$, strictly increasing on $[a,b]$, and of degree  $m$, it is always possible to find a generalized
Bernstein operator fixing $\mathbf{1}$ and $f_1$, on the
space $\mathbb{P}_{n}[a,b]$ with the standard Bernstein basis, 
provided that  $n\ge m$ and that $n$ is ``sufficiently large".  The special case $f_1(x) = x^j$,
$[a,b] = [0,1]$, had been previously solved in \cite[Proposition 11]{AKR08b};
on the other hand, it is known that such operators do not exist if we are
required to fix $f_0(x) = x^i$ and $f_1(x) = x^j$ on $[0,1]$, $1 \le i <  j$,
cf. \cite[Theorem 2.1]{Fi}. Regarding the 
meaning of ``sufficiently large", in general $n$ and the degree $m$ are not
comparable, i.e., there is no constant $C >0$ such that 
 $B_{n}^{f_1}$ is well defined for all $n \ge C m$.  We shall see  that
 for the family 
$f_{1,t} (x) = ( x - t)^3$ on $[0,1]$, where $t\in (0, 1/2)$, we must have
$n  > 1/t$ (cf. Theorem \ref{degree3}).

Denote by $\gamma_{n,k}$, $k = 0, \dots, n$ the  coordinates of $f_1$
with respect to
the Bernstein bases. Since we want to fix $f_1$, the nodes must be given
by $t_{n,k} = f_1^{-1} (\gamma_{n,k})$,  the injectivity of $f_1$ being 
used at this point: uniqueness of the coordinates entails that the
nodes are also uniquely determined. However, it need not be true 
  that all 
$f_1^{-1} (\gamma_{n,k}) \in [a,b]$,  (cf. Theorem \ref{degree3}). 
We prove that if $n\ge m$
is sufficiently large 
then indeed all   the coordinates 
$\gamma_{n,k}$ belong to $ [f_1(a), f_1 (b)]$. The proof proceeds by showing 
 that we have
$\left|\gamma_{n,k} - f_1\left( a+\frac{k}{n}\left( b-a\right)
\right)\right| = O(1/n)$, where the implicit constant in the big
$O$ notation depends only on $f_1$ over the interval $[a,b]$, but not on $k$ nor $n$ (cf. 
Theorem \ref{unifapprox}).

A considerable difference arises between the cases where $f_1^\prime$ vanishes
at some point inside $(a,b)$, and where $f_1^\prime >0$
on $(a,b)$: under the latter condition, the nodes form an increasing sequence and
the separation between consecutive nodes is bounded by $O(1/n)$, while
if $f_1^\prime (x_0) = 0$ for some $x_0 \in (a,b)$, then there are 
reversals of the nodes no matter how big $n$ is, 
and the difference between consecutive  nodes
can be distinctly larger than $O(1/n)$ (cf. Theorem \ref{Thm1},  Theorem \ref{K/n}, and Example  \ref{far}). When $f_1^\prime$ vanishes somewhere inside $(a,b)$,
letting $s - 1$ be the largest order of all such zeros, we are able
to prove that the separation between consecutive nodes is bounded by $O(1/n^{1/s})$
(cf. Theorem \ref{sHolder}), so despite the order reversal of some nodes, we
still have convergence to the identity.

Let $\omega \left( f, t\right)$ stand for the modulus of continuity of
a uniformly continuous function $f$.
A consequence of the preceding  results is that when $f_1^\prime >0$,
	\begin{equation*}
\left\vert  B_{n}^{f_{1}}f \left( x\right) -
B_{n}f \left( x\right) \right\vert \leq 
\omega \left( f, K n^{- 1}\right)
\end{equation*} 
for some constant $K > 0$,
while if $s - 1$ is the largest order of all the zeros of  $f_1^\prime$
inside $(a,b)$, then 
\begin{equation*}
\left\vert  B_{n}^{f_{1}}f \left( x\right) -
B_{n}f \left( x\right) \right\vert \leq 
\omega \left( f, K^\prime n^{- 1/s}\right)
\end{equation*} 
for some  $K^\prime > 0$.
Recalling that there is a
  $C  > 0$ (which depends on $[a,b]$ only) such that
\begin{equation*}
\left\vert  f \left( x\right) -
B_{n}f \left( x\right) \right\vert \leq 
C \omega \left( f,  n^{- 1/2}\right),
\end{equation*} 
 we see that up to some constant and for $s\le 2$,
the rate of approximation of the generalized Bernstein operators in terms
of the modulus of continuity, is no significantly worse than that for the classical operator: 
\begin{equation*}
	\left\vert  f \left( x\right) -
	B_{n}^{f_1}f \left( x\right) \right\vert \leq 
	(K^\prime + C)  \ \omega  \left( f,  n^{- 1/2}\right).
\end{equation*}

Now let us place the preceding results in context.
Several variations of the Bernstein operators have been considered in the
literature to address the problem of fixing functions other than
$\mathbf{1}$ and $e_1 ( x) = x$, sometimes modifying the Bernstein
bases functions 
(consider, for instance, the nowadays called King's operators,
after \cite{Ki}). 
Also, similar questions have been asked about related positive
operators, (cf. for instance \cite{AcArGo}).

 Within the line of research followed here, previous work
has focused
on spaces different or more general than spaces of polynomials (cf. \cite{MoNe00}, \cite{AKR07},
\cite{AKR08b}, \cite{AKR08}, \cite{KR07b}, \cite{AR}, \cite{Ma}),  but 
convergence has also been studied
in M\"untz spaces, cf. \cite{AiMa}, and for rational Bernstein operators, 
see \cite{Rend14}.

Generally speaking, the situation regarding existence is well understood in the context of
extended Chebyshev spaces (which  generalize
the space of polynomials of degree at most $n$, by retaining the bound on the
number of zeros) cf. \cite{AKR08b}: One considers a two
dimensional extended Chebyshev space $U_1$,  for which a generalized
Bernstein operator fixing it can always be defined, 
and inductively, this  definition is extended to
$U_1 \subset U_2 \subset \cdots \subset U_n$, where  each $U_k$ is a
$k + 1$-dimensional extended Chebyshev space.

While the present paper returns to the classical polynomial spaces, its results
go beyond  the  setting of chains of extended Chebyshev spaces (starting with
dimension two)
for oftentimes such chains, fixing $\mathbf{1}$ and $f_1$,  are  impossible to generate.

Understanding the polynomial case is a natural starting point 
towards possible generalizations to other spaces of functions, and
it sheds light on the usefulness of different notions
regarding generalized Bernstein operators.
For instance, if one requires in the definition that the sequence of nodes be non-decreasing (as done in \cite{Ma}) this will lead to  better  properties of the operators from the viewpoint of shape preservation (see the example at the end of the paper) but existence will never be obtained 
if $f_1^\prime$ has a zero in $(a,b)$.

The authors are indebted to two anonymous referees for carefully and quickly reviewing this paper. Their
efforts have allowed us to remove  several errors and typos from the original manuscript.

\section{Definitions and background.}

\begin{definition} Let $U_{n}$  be an $n+1$ dimensional subspace of $C^{n}\left( \left[
	a,b\right], \mathbb{R}\right)$
	(in this paper we consider real valued functions only). A 
	Bernstein basis $\{p_{n,k}: k=0,\dots,n\}$ of $U_n$ is a basis
	with the property that each $p_{n,k}$ has a zero of order $k$ at $a$, and
	a zero of order 
	$n-k$ at $b$. The function $p_{n,k}$ might have additional zeros inside 
	$\left(a,b\right) $; this is not excluded by the preceding definition. A Bernstein basis is {\em non-negative} if for all
	$k= 0, \dots, n$, $p_{n,k} \ge 0$ on  $\left[ a,b\right]$, and
	{\em positive} if $p_{n,k} > 0$ on $\left(a,b\right)$. Finally,
	a non-negative Bernstein basis is {\em normalized} if $\sum_{k=0}^n p_{n,k} \equiv 1$.
\end{definition} 

It is easy to check that non-negative Bernstein bases are unique
up to multiplication by a positive scalar, and that normalized
Bernstein bases are unique.

\begin{definition} If $U_{n}$ has a non-negative Bernstein basis $\{p_{n,k}: k=0,\dots,n\}$, we define a
	{\em generalized Bernstein operator} $B_{n}:C\left[ a,b\right] \rightarrow U_{n}$
	by setting 
	\begin{equation}
	B_{n}\left( f\right) =\sum_{k=0}^{n}f\left( t_{n,k}\right) \alpha
	_{n,k}p_{n,k},  \label{eqBern}
	\end{equation}
	where the nodes $t_{n,0},...,t_{n,n}$ belong to the interval 
	$\left[ a,b\right]$, and the weights $\alpha_{n,0},...,\alpha_{n,n}$ are positive. 
\end{definition} 

We briefly comment on the rather weak assumptions made in the preceding
definition. Non-negativity of the functions $p_{n,k}$ and positivity of the weights $\alpha_{n,0},...,\alpha_{n,n}$ are required
so that the resulting operator is  positive, a natural property from the viewpoint of
shape preservation. Strict positivity
of the weights entails that all the basis functions are used in
the definition of the operator.
Additionally, the nodes must belong to $\left[ a,b\right]$. This is a natural condition, since in
principle the domain of definition of the functions being
considered is $\left[ a,b\right]$.  Note that no requirement is made in the preceding
definition about the
ordering of the nodes, and in particular, we do not ask that they  be {\em strictly  increasing},
i.e., 
that  $t_{n,0} < t_{n,1} < \cdots < t_{n,n}$. When we only have
$t_{n,0} \le t_{n,1} \le \cdots \le t_{n,n}$ we say that the sequence of nodes is
{\em increasing}, or equivalently, {\em non-decreasing}.

\vskip .2 cm

The problem of
existence, as studied in  \cite{AKR08b} and \cite{AKR08}, arises when we choose two functions $f_{0},f_{1}\in U_{n}$, such that $f_{0} >0 $,  
$f_{1}/f_{0}$ is strictly increasing, and we  require that 
\begin{equation}
B_{n}\left( f_{0}\right) =f_{0}\text{ and }B_{n}\left( f_{1}\right) =f_{1}.
\label{eqBern2}
\end{equation}
If these equalities can be satisfied, they uniquely determine the location of the
nodes and the values of the coefficients, cf. \cite[Lemma 5]{AKR08b}; in other words, there is at most
one Bernstein operator $B_{n}$ of the form (\ref{eqBern}) satisfying (\ref
{eqBern2}).  
We will consistently use the following notation. Assume that $p_{n,k},$ 
$k=0,...,n$, is a Bernstein basis of the space $U_{n}$. Given 
$f_{0},f_{1}\in U_{n}$, there exist coefficients $\beta_{n,0},...,\beta_{n,n}$
and $\gamma_{n,0},...,\gamma_{n,n}$ such that 
\begin{equation}
f_{0}\left( x\right) =\sum_{k=0}^{n}\beta_{n,k}p_{n,k}\left( x\right) \text{
and }f_{1}\left( x\right) =\sum_{k=0}^{n}\gamma_{n,k}p_{n,k}\left( x\right) .
\label{eqeq}
\end{equation}

The following elementary fact regarding bases will be used throughout
(cf. \cite[Lemma5]{AKR08b}):
suppose there exists a generalized Bernstein operator $B_n$ of the form
 (\ref{eqBern}), fixing $f_0$ and $f_1$;
then it must be the case that for each $k= 0, \dots, n$, 
\begin{equation}
\beta_{n,k} = f_0 (t_{n,k}) \ \alpha_{n,k} \text{ \ \ \ 
	and \ \ \ } \gamma_{n,k}  = f_1 (t_{n,k}) \  \alpha_{n,k}.
\label{bases}
\end{equation}
If $f_0 =\mathbf{1} $, using $\tilde p_{n,k} := 
\alpha_{n,k}p_{n,k}$ instead of $p_{n,k} $,  we may assume that the Bernstein basis is normalized, 
and then we can take the coordinates of $\mathbf{1} $ and the weights to
be 1. Thus,  we have that
\begin{equation} \label{one}
1 = \alpha_{n,k} = \beta_{n,k}
 \text{ \ \ \ 
	and \ \ \ } t_{n,k}= f_{1}^{-1}\left(\gamma_{n,k}\right).
\end{equation}
In this case, we denote the generalized Bernstein operator by $B_{n}^{f_1}$.

\section{Characterizing when nodes  increase.}

For the remainder of the paper,   $p_{n,k}\left( x\right) $ will denote the usual Bernstein basis function 
\begin{equation} \label{standard}
p_{n,k}\left( x\right) =\binom{n}{k}\frac{\left( x-a\right) ^{k}\left(
b-x\right) ^{n-k}}{\left( b-a\right) ^{n}}
\end{equation}
on  $\mathbb{P}_n [a,b]$, the space of polynomials 
on $[a,b]$, of degree bounded
by $n$.

While Theorem \ref{Thm1} and its corollary
can be presented in greater generality, in order to minimize technicalities  we
shall restrict ourselves to the  polynomial spaces $\mathbb{P}_n [a,b]$
with their standard Bernstein bases, which is all we shall need in this  paper.

The following lemma is well known.

\begin{lemma} \label{lincreasing}
Let  $f_{1}$ be a  polynomial on $\left[ a,b\right] ,$ of degree bounded by $n \ge 1$,
with coordinates given by
$
f_{1}\left( x\right) =\sum_{k=0}^{n}\gamma _{n,k}p_{n,k}\left( x\right),
$
and let  $\frac{d}{dx}
f_{1} (x) =\sum_{k=0}^{n-1}w_{k}p_{n-1,k} (x)$. 
 For $k=1,\dots , n,$ we have  
 $$
  \gamma_{n,k}- \gamma_{n,k-1}  = \frac{(b -a )  w_{k - 1}}{n} .
  $$
\end{lemma}

\begin{proof} Taking the derivative of the Bernstein
	bases functions, we get, for $k = 0$,
	\begin{equation}
	\frac{d}{dx}p_{n,0}(x )=  - \frac{n}{b - a} \  p_{n-1,0} (x), 
	\label{eqder0}
	\end{equation}
	for  $0 < k  < n$,
	\begin{equation}
	\frac{d}{dx}p_{n,k}(x ) =   \frac{n}{b - a} \ \left( p_{n-1, k - 1} (x) -  p_{n-1,k} (x)\right),
	\label{eqderk}
	\end{equation}
	and for $k = n$, 
	\begin{equation}
	\frac{d}{dx}p_{n,n}(x )=   \frac{n}{b - a} \  p_{n-1, n - 1} (x). 
	\label{eqder0}
	\end{equation}
	Now using the preceding expressions and rearranging terms we get
	\begin{equation}
	\sum_{k=0}^{n-1}w_{k}p_{n-1,k} (x)
	=
	\frac{d}{dx} f_{1}\left( x\right) 
	=
	\sum_{k=0}^{n}\gamma _{n,k} \  \frac{d}{dx} p_{n,k}(x ) 
	\label{posw}
	\end{equation}
		\begin{equation}
	=   \frac{n}{b - a} \ \left(   \sum_{k=1}^{n} (\gamma _{n,k} - \gamma _{n,k - 1}) p_{n-1,k -1} (x) \right), 
	\label{posw2}
	\end{equation}
	so $w_{k - 1} =   \frac{n}{b - a} \ \left(   \gamma _{n,k} - \gamma _{n,k - 1}\right)$.
\end{proof}

\begin{theorem} \label{Thm1}
Let  $f_{1}$ be a strictly
increasing polynomial on $\left[ a,b\right] ,$ of degree bounded by $n \ge 1$,
with coordinates given by
$
f_{1}\left( x\right) =\sum_{k=0}^{n}\gamma _{n,k}p_{n,k}\left( x\right).
$
Then the following  are equivalent:

a)  There exists a generalized  Bernstein operator $B_{n}^{f_1}:C\left[ a,b\right]
\rightarrow \mathbb{P}_n\left[ a,b\right] $, fixing $\mathbf{1}$ and $f_1$, and defined by  
\begin{equation}
B_{n}^{f_1}f\left( x\right) =\sum_{k=0}^{n}f\left( t_{n,k}\right)p_{n,k}\left( x\right),  \label{defBPab}
\end{equation}
where $a= t_{n, 0} \le \dots \le t_{n,n} = b$ 
(resp.  $a= t_{n, 0} < \dots < t_{n,n} = b$).

b) For $k=0,\dots , n,$ the coefficients $\gamma_{n,k}$,
are increasing (resp. strictly increasing).

c) For $k=0,\dots , n-1,$ the coefficients $w_{k}$  defined by $\frac{d}{dx}
f_{1}=\sum_{k=0}^{n-1}w_{k}p_{n-1,k}$,  are non-negative (resp. strictly positive).
\end{theorem}

\begin{proof} The equivalence between a) and b) is immediate from (\ref{one}),
where it is noted that if $f_0 = \mathbf{1}$, then $\alpha_{n,k} = \beta _{n,k} = 1$ for $k=0, \dots, n$. Thus, if $B_{n}^{f_1}$ fixes $f_1$ and 
the nodes
are increasing with  $k$, since 
$f_1 (t_{n,k}) = \gamma_{n,k}$, the coordinates $\gamma _{n,k}$ also increase,
while if the coordinates increase, so do the nodes given by 
$t_{n,k} = f_1^{-1} (\gamma_{n,k})$, and, provided that the nodes belong to
$[a,b]$,  the operator defined by (\ref{defBPab})
clearly fixes $f_1$  and $ \mathbf{1}$.  So it is enough to check that
$t_{n,0} = a$ and  $t_{n,n} = b$. But this is obvious, since
 $f_{1}\left( a\right) = \gamma _{n,0}  \ p_{n,0} \left( a\right) =  \gamma _{n,0}$,
  and  $f_{1}\left( b\right) = \gamma_{n,n} \ p_{n,n}\left( b\right) =  \gamma _{n,n} .$ The strictly increasing case is identical.

Regarding  the equivalence of b) and c), by the preceding Lemma,
$
  \gamma_{n,k}- \gamma_{n,k-1}  = \frac{(b -a )  w_{k - 1}}{n}
  $, so for $k=1,\dots , n,$
$w_{k - 1}  \ge 0$ (resp. $w_{k - 1}  > 0$)
if and only if  $\gamma _{n,k} \ge  \gamma _{n,k - 1}$ 
(resp. $\gamma _{n,k} >  \gamma _{n,k - 1}$).
\end{proof}

Thus, we obtain the following  necessary condition for the existence of a
generalized Bernstein operator defined via an increasing 
sequence of nodes, with respect to the standard Bernstein
basis on $\mathbb{P}_n [a,b]$. 

\begin{corollary} \label{increasing}
Suppose that there exists a generalized  Bernstein operator 
 $B_{n}^{f_1}:C\left[ a,b\right]
\rightarrow \mathbb{P}_n\left[ a,b\right] $ with increasing nodes,
fixing the constant function $\mathbf{1}$ and a strictly increasing polynomial 
$f_{1}.$ Then $f_{1}^{\prime }\left( x\right) >0$ for all $x$ in the open
interval $\left( a,b\right) .$ If the sequence of nodes is strictly
increasing, then $f_{1}^{\prime } >0$ on the closed
interval $\left[a,b \right].$
\end{corollary}

\begin{proof}
If the sequence of nodes is increasing, by Theorem \ref{Thm1}, there are non-negative coefficients $
w_{k},k=0,\dots ,n-1,$ such that
$
f_{1}^{\prime }\left( x\right) =\sum_{k=0}^{n-1}w_{k}p_{n-1,k} (x).
$
Since $f_{1}^{\prime }$ is not identically zero, at least one of the coefficients is strictly positive, say $w_{j} > 0$.
And since  $p_{n-1,j} > 0$  on $\left( a,b\right) $, we also have
$f_{1}^{\prime } >0$  on $\left( a,b\right) .$ 

If the sequence of nodes is strictly increasing, then all the
coefficients $w_{k}$ are strictly positive. Since the standard Bernstein
basis on $\mathbb{P}_{n - 1} [a,b]$ forms a partition of unity, for all $x \in [a,b]$, 
$
f_{1}^{\prime }\left( x\right) =\sum_{k=0}^{n-1}w_{k}p_{n-1,k}  \ge
 \min_{k= 0, \dots, n-1} w_k \mathbf{1} (x) > 0
$.
\end{proof}

\begin{corollary} Fix  $k \ge 1$ and $a < 0 < b$. There is no generalized 
	Bernstein
	operator $B_{n}^{f_1}: C\left[ a,b\right] \rightarrow 
	\mathbb{P}_{n}\left[ a,b\right] $ fixing $\mathbf{1}$ and $f_{1} := x^{2 k + 1}$, and
	 defined by a non-decreasing sequence of nodes.
\end{corollary}

\section{Strictly positive polynomials have positive Bernstein coordinates,
	eventually.}

The title of this section recalls an old result of S. Bernstein, cf. 
\cite{Be1}, \cite{Be2}, according to which a polynomial $g > 0$ on 
$[a,b]$ has positive Bernstein coordinates in $\mathbb{P}_N [a,b] $, 
provided $N\ge \operatorname{deg} (g)$ is large enough. The case 
$g > 0$ on 
$(a,b) $ reduces to the previous one by factoring the zeros (if any) at the
endpoints, and then we conclude that the Bernstein coordinates are non-negative.
Note however that if $g \ge  0$  on 
$[a,b] $ and $g(c) = 0$ for some $c  \in 
(a,b) $, then some coordinate of $g$ must be negative, since the
Bernstein bases functions are positive on 
$(a,b) $.

Two  additional proofs of Bernstein's  result can be found in the answers to 
\cite[Problem 49 of Part VI ]{PoSz}.  It also follows from Theorem \ref{unifapprox} below,
which gives a proof analogous to Bernstein's original, the difference being that we
are interested in uniform estimates.

Let us remind the reader
of the fact that strictly positive polynomials can have some negative Bernstein coordinates, so the ``sufficiently large" clause is needed. 
The next example (essentially, the Example from \cite[Pg. 4684]{PoRe})  illustrates the fact that the 
necessary condition  ``$f_{1}^{\prime }  >0$ on
$\left( a,b\right)$" from Corollary \ref{increasing} to ensure
that nodes are non-decreasing, 
is not sufficient. In fact, 
even the assumption  
$f_{1}^{\prime }  >0$ on
$\left[ a,b\right]$ is not sufficient. 

\begin{example}  For $n = 2$ and 
	$[a,b] =[0,1]$, the
positive function 
	$f_1^\prime (x) := (x - 1/2)^2 + 1/8$ satisfies
	$f_1^\prime (x) 
	= 3 p_{2,0} (x) /8 -  p_{2,1} (x) /8 +  3 p_{2,2} (x) /8$.
	Consider the primitive with constant term zero,
	given by  $f_1 (x) = 3 x /8 - x^2/2 + x^3/3$. By Theorem \ref{Thm1}, if
	a generalized Bernstein operator fixing
	$\mathbf{1}$ and $f_1 (x)$  exists, then it is not defined via an increasing sequence
	of nodes. 
	Actually, in this case the operator $B_3^{f_1}$ does exist:
	 by computing coordinates we find that 
	$f_1 (x) =  p_{3,1} (x)/ 8 +  p_{3,2} (x)/12 +  5 p_{3,3} (x)/24$,
	and  since $f_1 (0) = 0$ and $f_1 (1) = 5/24$,
	all the nodes belong to $[0,1]$;
	we have $0 = t_{3,0} < t_{3,2} = f_1^{-1}(1/12) < t_{3,1} = f_1^{-1}(1/8)
	< t_{3,3} = 1$. We shall see
	 that for some $N > 3$ the nodes become disentangled, forming an
	increasing sequence. And once they become increasing, they stay that
	way, by Theorem \ref{Thm1} and  the next result,  which is a direct consequence of degree elevation. 
\end{example}

\begin{proposition}
	If $g\left( x\right) =\sum_{k=0}^{n}w_{n,k}p_{n,k}\left( x\right) $ on $\left[ a ,b\right] $ has non-negative coefficients $w_{n,k}$, then the
	Bernstein coefficients $w_{N,k}$ of $g\left( x\right) $ with respect to $
	p_{N,k}$, $k=0,...,N,$ for any
	$N\geq n$, also satisfy $w_{N,k}\ge 0$.
\end{proposition}

\begin{proof}
	By  an induction argument it suffices to prove it for $N=n+1.$ 
	Since  
	$$p_{n,k} = a_{n + 1, k} p_{n+1,k} +  a_{n + 1, k+ 1} p_{n+1,k+1},
	$$
	with constants $a_{n + 1, k} > 0$ and  $ a_{n + 1, k+ 1} > 0$,  
	\begin{eqnarray*}
		g\left( x\right)
		&=& \sum_{k=0}^{n}w_{n,k}p_{n,k}\left( x\right) = \sum_{k=0}^{n}w_{n,k}  a_{n + 1, k}  p_{n+1,k}+\sum_{k=0}^{n}w_{n,k}  a_{n + 1, k+ 1}  p_{n+1,k+1} \\
		&=&w_{n,0} a_{n + 1, 0} p_{n+1,0} + 
		\sum_{k=1}^{n}\left( w_{n,k}+w_{n,k -1}\right)  a_{n + 1, k} 
		p_{n+1,k} + w_{n,n}  a_{n + 1, n+ 1} p_{n + 1,n + 1}.
	\end{eqnarray*}
\end{proof}

\section{Generalized Bernstein operators fixing $\mathbf{1}$ and a strictly increasing  polynomial.}

In this section we prove that for $n$ sufficiently large, there is a generalized Bernstein
operator fixing $\mathbf{1}$ and a nonconstant polynomial $f_1$ under the  assumption 
$f_1^\prime \ge 0$ on $[a,b]$. More precisely, we prove the following theorem, 
presenting some preliminary results before the proof.

\begin{theorem} \label{Thm3} Given any non constant polynomial	$f_{1}$
	with $f_{1}^{\prime }  \ge 0$ on
	$\left[ a,b\right]$, for every  $N$  sufficiently large  there
	is a generalized Bernstein operator $B_{N}^{f_1}$ based on $\mathbb{P}_N [a,b] $,
	which fixes both $\mathbf{1}$ and $f_1$. If there exists an $x_0\in (a,b)$ 
	such that $f_{1}^{\prime } (x_0) = 0$, the sequence of nodes will fail
	to be non-decreasing for all $N$ such that $B_{N}^{f_1}$ exists. Fix $N$ large 
	enough 
	so that $B_{N}^{f_1}$ is well defined. If $f_{1}^{\prime }$
	has a zero of order $s_1$ at $a$, the first $s_1 + 1$ nodes are equal to
	$a$, and if $f_{1}^{\prime }$
	has a zero of order $s_2$ at $b$, the last $s_2 + 1$ nodes are equal to
	$b$.  If $f_{1}^{\prime }  > 0$ on
	$\left(a,b\right)$, then the nodes in 	$\left(a,b\right)$ form a strictly
	increasing sequence, while if $f_{1}^{\prime }  > 0$ on
	$\left[ a,b\right]$, then all nodes form a strictly increasing sequence.
\end{theorem}

Let us recall the well known relationship
between the Bernstein bases and the monomial or power bases. 

\begin{proposition} \label{coordinates}
Let $g:\left[ a,b\right] \rightarrow \mathbb{R}$ be a polynomial of degree 
$
m $ and let $n\geq m$ be a natural number. Write 
\begin{equation*}
g\left( x\right) =\sum_{l=0}^{m}c_{l}\left( x-a\right)^{l}
 =\sum_{k=0}^{n}w_{n,k}p_{n,k}\left( x\right).
\end{equation*}
Then the Bernstein coefficients $w_{n,k}$ of $g$ are given by 
\begin{equation} \label{Berncoor}
w_{n,k}=\sum_{l=0}^{\min \left( k,m\right) }c_{l}
\frac{k!\left( n-l\right) !}{ n!  \left( k-l\right) !}  \left( b-a\right)^{l}.
\end{equation}
\end{proposition}

\begin{proof}
Since $1=\sum_{k=0}^{n-l}p_{n-l,k}\left( x\right) $,
\begin{equation*}
g\left( x\right) \cdot 1=\sum_{l=0}^{m} c_{l}  \left( x-a\right)^{l} \sum_{k=0}^{n-l}p_{n-l,k}\left( x\right)
=
\sum_{l=0}^{m}\sum_{k=l}^{n}c_{l}  \left( x-a\right)^{l} p_{n-l,k-l}\left(x\right).
\end{equation*}
Now
\begin{equation}
 \left( x-a\right)^{l} p_{n-l,k-l}\left( x\right) 
= 
\binom{n-l}{k-l} \frac {\left( x-a\right)^{k}\left(b -x\right) ^{n-k}}
{ \left( b-a\right)^{n -l}}
=
\frac{k!\left( n-l\right) !}{\left( k-l\right) !n!}
 \  \left( b-a\right)^{l}  p_{n,k}\left( x\right), 
  \label{eqdegree}
\end{equation}%
so
\begin{eqnarray*}
g\left( x\right) &=&\sum_{l=0}^{m}\sum_{k=l}^{n}c_{l}\frac{k!\left(
n-l\right) !}{\left( k-l\right) !n!} \  \left( b-a\right)^{l}  p_{n,k}\left( x\right)\\
&=&
\sum_{k=0}^{n} p_{n,k}\left( x\right) \sum_{l=0}^{\min \left( k,m\right)
}c_{l}\frac{k!\left( n-l\right) !}{ n!  \left( k-l\right) ! } \  \left( b-a\right)^{l}.
\end{eqnarray*}
\end{proof}

How large must $N$ be so that $B_{N}^{f_1}$ is well defined cannot be determined
from the degree of $f_1$ alone, it also depends on the coefficients $c_l$.

\begin{theorem} \label{degree3}   Set $f_{1} (x) = ( x - N^{-1})^3$ on $[0,1]$.	
Then the node $t_{N, 2} < 0$, so $B_{N}^{f_1}$ is not well defined.
\end{theorem}

\begin{proof} Since
	$$
	f_1(x) = 
	\left( x-\frac{1}{N}\right) ^{3}
	=
	-\frac{1}{N^3}  + \frac{3 x }{N^2 }-\frac{3 x^{2}}{N } + x^3
	$$
	on  $\left[ 0,1\right]$,  from Proposition \ref{coordinates} below 
	we know that the Bernstein 
	coordinate $\gamma_{N,2}$ is given by 
		\begin{equation}
	\gamma_{N,2}=\sum_{l=0}^{2}c_{l}\frac{2!\left( N-l\right) !
	}{N!\left( 2-l\right) !}
	=
	-\frac{1}{N^3}+\frac{6}{N^3}-\frac{6}{N^2 \left(N -1\right) } 
	=  \frac{5}{N^3}-\frac{6}{N^3  - N^2}.
	\end{equation}
	Now  $f_1^{-1} (y) = 1/N + y^{1/3} $ is increasing, 
	so 
	$$
	t_{N,2} = f_1^{-1} \left(\frac{5}{N^3}-\frac{6}{N^3  - N^2} \right)  
	<  
	f_1^{-1} \left(\frac{5}{N^3}-\frac{6}{N^3} \right) = 0. 
	$$
\end{proof}

Next we show that for an arbitrary polynomial $g$, restricted to the interval $[a,b]$,  we have 
$\left\vert g\left( a+\frac{k}{n}\left( b-a\right) \right)
-w_{n,k}\right\vert  
=  
O (1/n)$, where the  constant 
implicit in the  order notation depends only
on $g$ and on $[a,b]$, but not on $n$ or $k$. In particular, the following proof entails
that if $g$ is affine, then $ g\left( a+\frac{k}{n}\left( b-a\right) \right)
 = w_{n,k}$, that is, we obtain the well known fact that
 the standard Bernstein operators fix the affine functions. So the result says
 something new only  when $m = \deg (g) \ge 2$.

Write $e_j (x) := x^j$,  and
let $n \gg  1 $ be even, say $n = 2k$. We know from Proposition \ref{coordinates}
that for $e_2(x)$ on $[0,1]$, $
w_{n,k}
=
\frac{k\left( k - 1 \right) }{ n \left( n- 1 \right)  }
$, so 
$|e_2(k/n) - w_{n,k} | = (4 n - 4 )^{ - 1 }$.
 Thus,  the estimate 
$|g (k/n) - w_{n,k} |  =  O (1/n)$ cannot in general be improved.

We use the standard conventions whereby the value of an empty sum
is 0, and the value of an empty product, 1.

\begin{theorem} \label{unifapprox}
Let $g\left( x\right) =\sum_{l=0}^{m}c_{l}  (x - a)^{l}$ be a polynomial of degree 
$m $, restricted to the interval $[a,b]$. Define $c_{\operatorname{max}} :=   \max_{l=2,\dots ,m}\left\vert c_{l}\right\vert$.
 For $n > m$  and $k=0,\dots , n$, let   $w_{n,k}$ be the Bernstein coefficients of 
$g$ in $\mathbb{P}_n [a,b]$, so 
\begin{equation*}
	g\left( x\right) =\sum_{k=0}^{n}w_{n,k}\binom{n}{k}\frac{\left( x-a\right)
		^{k}\left( b-x\right) ^{n-k}}{\left( b-a\right) ^{n}}.
\end{equation*}
Then 
\begin{equation} \label{images}
\left\vert g\left( a+\frac{k}{n}\left( b-a\right) \right)
-w_{n,k}\right\vert
\leq 
\frac{m^{3} c_{\operatorname{max}} \max\{  \left(b-a\right) ^{2},  \left(b-a\right) ^{m}\} }{n}.
\end{equation}
\end{theorem}

\begin{proof}  By (\ref{Berncoor})
\begin{equation}
g\left( a+\frac{k}{n}\left( b-a\right) \right) - w_{n,k}
\label{eqwww}
\end{equation}
\begin{equation}
=
\sum_{l=0}^{\min \left( k,m\right) }
c_{l} \left( \left( \frac{k}{n}\right)^{l} - 
\frac{k!\left( n-l\right) !}{ n!  \left( k-l\right) ! }\right) 
\left(b-a\right)^{l}
+
\sum_{l=k + 1}^{m} c_{l}\left(\frac{k}{n}\right) ^{l} \left(b-a\right)^{l}.
\label{eqwwwb}
\end{equation}
Since for $l = 0, 1$, $\left( \frac{k}{n}\right) ^{l} - 
\frac{k!\left( n-l\right) !}{ n!  \left( k-l\right) ! }\ = 0$, we can start the sums with $2 = l \le  m$ (in particular, if $m < 2$, then $g\left( a+\frac{k}{n}\left( b-a\right) \right) = w_{n,k}$).
 Thus 
\begin{equation*}
\left| g\left(  a+\frac{k}{n}\left( b-a\right)  \right) - w_{n,k}\right|
\end{equation*}
\begin{equation*}
=
\left|\sum_{l=2}^{\min \left( k,m\right) }
c_{l} \left( \left( \frac{k}{n}\right) ^{l}-\frac{k\left(
	k-1\right) \cdots \left( k-\left( l-1\right) \right) }{n\left( n-1\right)
	\cdots \left( n-\left( l-1\right) \right) }\right) \left(b-a\right)^{l}
+
\sum_{l=k + 1}^{m} c_{l}\left(\frac{k}{n}\right) ^{l} \left(b-a\right)^{l}\right|.
\end{equation*}

Note  that for every $0 \le s < n$, $\frac{k}{n}\geq \frac{k-s}{n-s}$, as can be checked just by 
simplifying. Furthermore, if $b - a \ge 1$, then for  $2 \le l \le  m$ we have $\left(b-a\right)^{l} \le \left(b-a\right)^{m}$, while  if $b - a \le 1$, then for  $2 \le l \le  m$ we have $\left(b-a\right)^{l} \le \left(b-a\right)^{2}$.
Write $M :=  \max\{  \left(b-a\right) ^{2},  \left(b-a\right) ^{m}\}$. Let us 
consider first the case $k \le m$. Then 
\begin{equation}
\left| g\left(  a + \frac{k}{n}\left( b-a\right) \right) - w_{n,k}\right|
\label{eqwww1}
\end{equation}
\begin{equation}
\le
\sum_{l=2}^{m }
|c_{l}| \left( \frac{k}{n}\right) ^{l}  \left(b-a\right)^{l}
\le
 Mc_{\operatorname{max}}  \left( \frac{m}{n}\right)^2
\sum_{l=0}^{\infty } \left( \frac{m}{n}\right)^l
\le
M \left( \frac{n}{n - m}\right) \frac{m^{2} c_{\operatorname{max}} }{n^2}.
\label{eqwww1b}
\end{equation}

Next, suppose $k \ge m$. In this case,
\begin{equation}
	\left| g\left(  a+\frac{k}{n}\left( b-a\right)  \right) - w_{n,k}\right|
	\le
	\sum_{l=2}^{ m }
	\left| c_{l} \right| \left( \left( \frac{k}{n}\right) ^{l}-\frac{k\left(
		k-1\right) \cdots \left( k-\left( l-1\right) \right) }{n\left( n-1\right)
		\cdots\left( n-\left( l-1\right) \right) }\right)  \left(b-a\right)^{l}
	\label{eqref}
\end{equation}
\begin{equation} 	\label{eqref1}
\le 
M c_{\operatorname{max}} \ 	\sum_{l=2}^{m }  
	\left( \left( \frac{k}{n}\right) ^{l}-
	\left( \frac{k - m}{n - m}\right) ^{l}\right).
	\end{equation}
	Using the mean value theorem for the function $h_{l}\left( x\right) =x^{l}$
	and the interval $\left[ \frac{k- m }{n- m },
	\frac{k}{n}\right] $, we have 
	\begin{equation} 	\label{eqref2}
		\left( \frac{k}{n}\right) ^{l}-\left( \frac{k-m }{n- m}\right) ^{l} 
	 \leq 
	 l \max_{\xi \in \left[ 
			\frac{k-m}{n-m},\frac{k}{n}\right] }\xi ^{l-1}
		\left( \frac{k}{n}-\frac{
			k- m }{n- m }\right)
			\end{equation}
			\begin{equation} 	\label{eqwww3}	
			\leq 
		m \left( \frac{k}{n}\right) ^{l-1}
		\frac{m \left( n-k\right) }{n\left( n- m\right) } 
		\leq 
		\frac{m^2}{n}.
	\end{equation}
	It follows that 
	\begin{equation} 	\label{eqwww2}	
	\left| g\left( a+\frac{k}{n}\left( b-a\right)  \right) - w_{n,k}\right|
\le
M 	c_{\operatorname{max}} \ 	\sum_{l=2}^{m }
	\frac{m^2}{n}
\le
	M	c_{\operatorname{max}} \   \frac{m^3}{n}.
\end{equation}
Since the estimate from (\ref{eqwww2}) is always larger than the estimate
from (\ref{eqwww1})-(\ref{eqwww1b}),
the result follows.
\end{proof}

\begin{remark}  The uniform (in $k$) approximation obtained in the preceding 
	Theorem depends on
	the degree of the polynomial, its coefficients, and the length of
	the interval, cf. (\ref{images}). The dependency on the last factor is easily explained:
	the smaller $b - a$ is, the larger the ``sampling rate" of the
	polynomial, and viceversa.
	
Also, from the preceding proof we see
 that near $0$ (say, for $0\le k \le m$) we have
$$
\left| g\left( a+\frac{k}{n}\left( b-a\right)  \right) - w_{n,k}\right|
=
O
\left( \frac{1}{n^2}\right), 
$$
 which of course is better than the general 	$
\left| g\left( a+\frac{k}{n}\left( b-a\right)  \right) - w_{n,k}\right|
=
O
\left( \frac{1}{n}\right). 
$ 
We can also get 
$
\left| g\left( a+\frac{k}{n}\left( b-a\right)  \right) - w_{n,k}\right|
=
O
\left( \frac{1}{n^2}\right)
$ at the other endpoint, for $n - m \le k \le n$, just by using 
$
\frac{m \left( n-k\right) }{n\left( n- m\right) } 
\leq 
\frac{m^2}{n(n - m)}
$
in 	(\ref{eqwww3}).

\end{remark}

The basic idea of the following proof is that the reversal of the nodes can only 
happen ``far away" from the endpoints, and then  by Theorem \ref{unifapprox}, 
for sufficiently large $N$, if $k/N$
is far away  from the endpoints of $[0,1]$, we must  have  $f_1 (a) < \gamma_{N,k} <  f_1 (b)$.

\vskip .3 cm

{\em Proof of Theorem \ref{Thm3}.} 
We assume, for simplicity in the expressions, that
$[a,b] = [0,1]$. There is no loss of generality in doing so, since
the increasing affine change of variables that maps $0$ to $a$ and
$1$ to $b$ preserves the non-negativity (resp. positivity) of $f_1^\prime$ on the 
closed interval (resp. on the open or on the closed interval). 

 Suppose  that $f_1^\prime$ has at least
one zero in $(0,1)$.  The case where $f_1^\prime > 0$ on $(0,1)$
is similar but simpler, since there are no reversals in the ordering of the nodes
(cf. Corollary \ref{increasing}).

 Let $m \ge 2$
be the degree of $f_1$.
 We need to show
that for $N$ sufficiently large, all nodes $t_{N,k}$  belong to $[0,1]$.

Suppose that $f_1^\prime$ 
has a zero of order $s_1 \ge 0$ 
at $0$, and a zero of order $s_2 \ge 0$
at $1$; let us write $f_1^\prime (x) = x^{s_1} g(x) ( 1- x)^{s_2}$,
with   $g = \sum_{k=0}^{n}
\widetilde{w}_{n,k}p_{n,k}$  for 
$n \ge m -1$. Thus
\begin{equation*}
f_1^\prime \left( x\right) 
=
\sum_{k=0}^{n}\widetilde{w}_{n,k} x^{s_{1}} p_{n,k} (x)
 \left( 1-x\right) ^{s_{2}}
\end{equation*}
\begin{equation*}
= 
\sum_{k=0}^{n}\widetilde{w}_{n,k} 
\frac{\binom{n}{k}}{\binom{n + s_1 + s_2}{k + s_1}} p_{n + s_1 + s_2 , k  + s_1} (x)
= 
\sum_{k=0}^{n + s_1 + s_2} {w}_{n + s_1 + s_2,k} p_{n + s_1 + s_2 ,k} (x).
\end{equation*}
Equating coefficients, we find that for $0 \le k \le s_1 - 1$ and for $n + s_1 +1 \le k \le n + s_1 + s_2$,
$ {w}_{n + s_1 + s_2,k} = 0$, while for $s_1 \le k \le n + s_1$, $\widetilde{w}_{n,k - s_1}$ and 
${w}_{n + s_1 + s_2 ,k}$
have the same sign.

Let 
$x_0, x_l \in (0,1)$ be the first and the last zeros of $f_1^\prime$ in $(0,1)$.
Since $f_1$ is increasing, we can select a  $\delta > 0$ 
 such that $f_1(0) <  f_1(x_0/2) - \delta <   f_1((x_l + 1)/2) + \delta
 < f_1(1)$. We  choose  $n_1$ such that
 $\frac{m^{3}	c_{\operatorname{max}} }{n_1} < \delta$. By 
 Theorem \ref{unifapprox}, applied to $f_1$, whenever $N \ge n_1$ and $k/N \le 
 (1 + x_l)/2$, we have $\gamma_{N,k} < f_1(1)$ and hence $t_{N,k} < 1$.
 Likewise, whenever $k/N \ge 
 x_0/2$, we have $\gamma_{N,k} > f_1(0)$ and $t_{N,k} > 0$. Thus, if
 $x_0/ 2 \le k/N \le (x_l + 1)/2$, then $t_{N,k} \in  (0,1)$.
 
 Using   Theorem \ref{unifapprox} again, this time applied to $g$, 
 we can select an even larger $N$ so that for every $k$ 
 with either $0 \le k/N \le 3x_0/4$ or $ (3 x_l + 1)/4 \le k/N \le 1$,
the coordinates $\widetilde{w}_{N - s_1 - s_2 , k} $ of $g$ in dimension $N - s_1 - s_2$ are strictly positive, and thus, the coordinates of $f_1^\prime$ in dimension  $N$ 
 satisfy  $w_{N,k} \ge 0$. 
From Lemma \ref{lincreasing} it follows that   $\gamma_{N,k} \le \gamma_{N,k + 1}$
for every such $k$. 
Since
$f_{1}\left( 0\right) = \gamma _{N,0}  \ p_{N,0} \left(  0 \right) =  \gamma _{N,0}$,
we have, for $0 \le k/N  \le 3x_0/4$, that the nodes $t_{N, k}
= f_{1}^{-1}\left( \gamma _{N,k} \right) 
\ge t_{N,0} = 0$. Likewise,  for $ (3 x_l + 1)/4 \le k/N \le 1$  we have $t_{N, k} \le t_{N,N} = 1$.

 Regarding the statements about the nodes, if $f_{1}^{\prime }  >0$ 
on
$\left[a,b \right]$, then for $N$ sufficiently large and
$k= 0, \dots , N$, the Bernstein coordinates
$w_{N, k}$ of $f_{1}^{\prime }$ satisfy  $w_{N, k} > 0$, 
so by Lemma \ref{lincreasing} or Theorem \ref{Thm1}, the coordinates $\gamma_{N, k}$ 
of $f_1$ form a strictly increasing sequence, and hence, so do
the nodes (their inverse image under $f_1$). Likewise, 
 by Lemma \ref{lincreasing},  $\gamma_{N, 0} = \cdots = \gamma_{N, s_1}$ when
 $f_1^\prime$ 
 has a zero of order $s_1$ 
 at $0$. The statements for the cases where $f_1^\prime$ 
 has a zero of order  $s_2$ at $1$, and where $f_{1}^{\prime }  >0$ 
 on
 $\left(a,b \right)$, follow in the same manner. Finally, the fact
 that nodes must decrease at some point if  $f_{1}^{\prime }$ vanishes
 somewhere inside $(0,1)$, is immediate from Corollary \ref{increasing}.
\qed

\section{Convergence to the identity.}

Next we show that the  generalized Bernstein  operators fixing constants and an increasing polynomial $f_1$,  denoted in this section by  $B_n^{f_1}$,  
converge to the identity in the strong operator topology, as $n\to \infty$.

The proof  proceeds 
as follows: We deduce the convergence of the operators  $B_n^{f_1}$ to the identity from the convergence
of the standard Bernstein operators $B_n$, by showing that the nodes 
$t_{n,k}$ of $B_n^{f_1}$
are ``close" to the corresponding nodes $a+\frac{k}{n}\left( b-a \right)$  of $B_n$, which follows from  the
fact   that
for $0 \le k \le n$, $f_1(t_{n,k})$ is ```close" to $f_1\left(a+\frac{k}{n}\left( b-a\right)\right)  $
(Theorem \ref{unifapprox}). 

We consider first the case $f_1^\prime > 0$ inside $(a,b)$, where estimates
are better for the following reason:  the zeros of $f_1^\prime$ inside 
$(a,b)$ generate reversals of the order
of the nodes. This back and forth movement entails that the average
distance between consecutive  nodes will be larger than $(b - a)/n$. And the
distance between a concrete pair of nodes can be much larger than $(b-a)/n$,
as the next example shows. However, when 
$f_1^\prime > 0$ inside $(a,b)$ there are no order reversals; thus, 
the average distance between consecutive nodes 
is still $\left( b-a\right)/n$. We show that
the distance between every pair of consecutive nodes is bounded by
$O(1/n)$, so in fact it is never much larger than its average
value. This has the obvious consequence that $B_n^{f_1} f (x)$ and 
$B_n f (x)$ are always at distance $O(\omega (f, 1/n))$.

\begin{example} \label{far} We use the same example $(x - t)^3$ as in Theorem \ref{degree3}, 
but taking $t = 1/2$ instead of $t = 1/N$.  Consider 
	$$
	f_1(x) = 
	\left( x-\frac{1}{2}\right) ^{3}
	=
	x^{3}-\frac{3}{2}x^{2}+\frac{3}{4}x-\frac{1}{8}
	$$
	on  $\left[ 0,1\right].$ Since the largest order $s-1$ of a zero of
	$f_1^\prime$ is 2, we have  $s = m = 3$.
	
	From Proposition \ref{coordinates} we know that for $n \ge k \ge m$, the Bernstein 
	coefficients of  a
	polynomial $g$  of degree $m$,  $g:\left[ 0,1\right] \rightarrow \mathbb{R}$,  
	\[
	g\left( x\right)
	=\sum_{l=0}^{m}c_{l}x^{l}=\sum_{k=0}^{n}w_{n,k}p_{n,k}\left( x\right),
	\]
	are given by 
	\begin{equation}
	w_{n,k}=\sum_{l=0}^{m} c_{l}\frac{k!\left( n-l\right) !
	}{n!\left( k-l\right) !}.  \label{Bernstcoor}
	\end{equation}
	In the specific instance $
	f_1(x) = 
	\left( x-\frac{1}{2}\right) ^{3},
	$ for $k\ge 3$ the coordinates $\gamma_{n,k}$ are 
	\begin{equation*}
	\gamma_{n,k} 
	=
	-\frac{1}{8}+\frac{3k}{4n}-\frac{3k\left( k-1\right) }{2n\left(n-1\right) }+\frac{k\left( k-1\right) \left( k-2\right) }{n\left( n-1\right)
		\left( n-2\right) }.
	\end{equation*}
	Let $K > 0$ be any fixed constant, and let 
	$N \gg  16^{  2} K^{3}$. When  $n=2N$ and $k=N$, we have 
	$\gamma_{2N,N}=0$, while if
	$n=2N$ and $k=N + 1$, a computation shows that  $\gamma_{2N,N+1} 
	=
	-\frac{3}{16N^{2}-8N} < 0$.
	Since $t_{n,k} =  
	f_1^{ -1 }(\gamma_{n,k})$ and  $f_1^{ -1 }(t)
	= 1/2 + t^{1/3}$, 
	we see that $t_{2N, N} = 1/2$,   $t_{2N, N + 1} = 1/2 - \left(\frac{3}{16N^{2}-8N} \right)^{1/3}$
	and 
	\[
	|t_{2N,N+1}-t_{2N,N}| =   \left(\frac{3}{16N^{2}-8N} \right)^{1/3} \ge (16 N)^{-2/3}
	\gg K N^{-1}.
	\]
	
	Also, the distance between $k/n$ and $t_{n,k}$ can be much larger than $ K n^{-1}$.
	Choose $N \gg 3 K^{2}$ so that both $N/2$ and $(N/3)^{1/2}$ are integers. A computation
	shows that 
	$\gamma_{N, N / 2 + (N/3)^{1/2}}  < 0$, so $t_{N, N / 2 + (N/3)^{1/2}}  < 1/2$ and 
	\[
	\left|t_{N, N / 2 + (N/3)^{1/2}} -  \frac{N + 2 (N/3)^{1/2}}{2 N}  \right| 
	\ge
	\frac{ (N/3)^{1/2}}{N}
	=
	\frac{ 1}{(3 N)^{1/2}}
		\gg \frac{ K}{N}.
	\]
\end{example}

\begin{theorem} \label{K/n} Let $f_{1}$ be a strictly increasing polynomial on $[a,b]$,
	of degree $m$,  such that $f_{1}^{\prime }\left( x\right) >0$ for all $x$ in the open
	interval $\left( a,b\right) .$ Then there exist an integer $n_0$ 
	and a constant $K > 0$ such that
	for all $n \ge n_0$,  and 
	all $0\leq k\leq n$,  we have   
	\begin{equation}  \label{small}
\left\vert a+\frac{k}{n}\left( b-a\right) -t_{n,k}\right\vert \leq
\frac{K}{n},
\end{equation}
	where $t_{n, k} = f_{1}^{-1}\left( \gamma_{n,k}\right)$, and the
	$\gamma_{n,k}$'s  are the Bernstein coordinates of $f_1$ in  $\mathbb{P}_{n}[a,b]$.
\end{theorem}

\begin{proof}
By a change of variables we  assume that $a=0$ and $b=1.$ 
This may alter the concrete value 
of $K$ in inequality (\ref{small}), since $b - a$ appears in (\ref{images}),
but it will not change the bound	$\left\vert a+\frac{k}{n}\left( b-a\right) -t_{n,k}\right\vert =
O\left(\frac{1}{n}\right).$
Choose $n \ge n_0 \gg m$ so large
	that all the Bernstein coordinates of  $f_{1}^{\prime }$ are non-negative 
	in all dimensions $n \ge n_0 - 1$. Additional conditions will be imposed on $n_0$
	later on. By Lemma \ref{lincreasing}  the coordinates of $f_{1}$, and hence the nodes,
	are non-decreasing for all $n \ge n_0$. We consider the case where $f_{1}^{\prime }$
	vanishes both at 0 and at 1. The other cases are simpler and can be handled in the
	same way. Suppose that  $f_{1}^{\prime }$
	has a
	zero of positive  order $s - 1$ at 0, so $s - 1> 0$ and $f_1 (x)  - f_1(0)= \sum_{l=s}^{m}c_{l}  x^{l}$. Then
	 $c_s > 0$ (since $f_1$ is increasing), and 
	$f_1^\prime (x) = \sum_{l=s}^{m}c_{l} l  x^{l -1} = s c_s x^{s - 1} + O(x^s)$. Thus, there exists a $\delta_0
	\in (0, 1/2)$ such that  for all 
	$y \in [0,  \delta_0)$, 
	\begin{equation} \label{der}
	\frac{c_{s} }{2} y^{s}
	\le 
	f_1 (y)  - f_1(0)
	\mbox{ \ \ \ and \ \ \ }
	\frac{s c_{s} }{2} y^{s - 1}
	\le 
	f_1^\prime (y).
	\end{equation} 
	Suppose next that  $f_1^\prime$ has a
	zero of order $r  - 1>  0$ at 1.  Recalling that $f$ is strictly increasing, we conclude that for $y < 1$, $f_1(y) - f_1(1) < 0$
	and $f^\prime_1 (y) >0$.  Thus, 
	there exists a $\delta_1
	\in (0, 1/2)$ such that  for all 
	$y \in (1 - \delta_1, 1]$, 
		\begin{equation} \label{der1}
	f_1 (y)   - f_1(1) =\sum_{l=r}^{m}c_{l}  (y - 1)^{l} =  c_r (y - 1)^{r} + O((y - 1)^{r + 1})
	\le	\frac{c_{r} }{2}( y  -1 )^{r} \le 0
	\end{equation} 
	and 
	\begin{equation} \label{der2}
	f_1^\prime (y) = \sum_{l=r}^{m}c_{l} l  (y - 1)^{l -1}
	 = 
	 r c_r (y - 1)^{r - 1} + O((y - 1)^r)
	 \ge
	 \frac{r c_{r} }{2}( y  -1 )^{r -1}
	\ge 0.
	\end{equation} 
	We write $[0,1]$ as the union of the three subintervals
	$[0, \delta_0/2)$, $[\delta_0/4 ,1 -  \delta_1/4]$ and $(1 - \delta_1/2 , 1]$;
	let 
	$$
	0 < c = \min_{x\in [\delta_0/4 , 1 - \delta_1/4]}  f_{1}^{\prime } (x).
	$$
	If both 
	$k/n$ and $t_{n,k}$ belong to  $[\delta_0/4,1 -  \delta_1/4]$,
	we use the fact that $f_1^{-1}$ is Lipschitz on the interval $f_1( [\delta_0/4, 1 - \delta_1/4])$, with constant 
	$1/c$,  to conclude
	that 
	$| k/n - t_{n,k} | 
	=
	O (n^{- 1})$. 
	
	Let  $n_0 \gg 2/( c \min \{\delta_0, \delta_1\})$.
Suppose either $k/n$ or $t_{n,k}$ belong to  $[0, \delta_0/2)$;  
	if either $k/n$ or $t_{n,k}$ belong to  $(1 - \delta_1/2, 1]$, 
	the argument is entirely analogous. 
	We must  have that both $k/n$ and $t_{n,k}$  belong to  $[0, \delta_0)$, 
	for otherwise   
	$$
	c \delta_0/2  
	\le
	 f_1 (\delta_0) - f_1 (\delta_0/2) 
	 \le
	 |f_1 (k/n) - \gamma_{n,k} |  = O (1/n)
	 $$
	 by Theorem \ref{unifapprox}, and  a
	  contradiction is obtained since $n \ge n_0$. 
	
	Let $k_0$ and $j$ be the smallest integers such that 
	\begin{equation} \label{kj}
	n \delta_0 \le 2 k_0 \mbox{ \ \ \ and \ \ \ }
	j \ge  2^s  (s c_s)^{-1} c_{\operatorname{max}} m^3
	\end{equation}
	 respectively, and assume that $k < k_0$, 
	 so   $k/n < \delta_0 /2$.	Note that $j > m$.
	First, since the nodes are increasing,  if $k\leq 2j$, then 
	\begin{equation} \label{klem}
	| t_{n,k} - k/n | \leq  t_{n,k} + k/n \leq  t_{n,2j}  +  2j/n
	\leq | t_{n,2j} -2j/n | + 4j/n.
	\end{equation}
	Thus, it  is enough to prove
	that for all $2j \le k < k_0$, 
	\begin{equation} \label{kgem}
	\frac{k - j}{n } \le t_{n,k}  \leq   \frac{k + j}{n },
	\end{equation}
	since this implies, whenever $0 \le k < k_0$, that
	\begin{equation} \label{klem}
	\left| t_{n,k} - \frac{k}{n} \right| \leq  
	\max \left\{ \frac{j}{n },   | t_{n,2j} -2j/n | + \frac{4 j}{n }\right\}
	\le \frac{5 j}{n }.
	\end{equation}
	So suppose $k \ge 2j > m$. Going back to formulas (\ref{eqref}) - (\ref{eqwww3}), the only difference
	here
	is that we have $s\le l \le m$ instead of $2\le l \le m$, 
	so using the bound $(k/n)^{l - 1} \le (k/n)^{s - 1}$ instead of 
	$(k/n)^{l - 1} \le 1$ in (\ref{eqwww3}), we obtain 
	the following refinement of (\ref{eqwww2}):
	\begin{equation}  \label{eqwww2r}
	\left| f_1\left( \frac{k}{n} \right) - \gamma_{n,k}\right|
	\le
	\left( \frac{k}{n} \right)^{s - 1} c_{\operatorname{max}} \ 	\sum_{l=s}^{m }
	\frac{m^2}{n}	
	\le
	\left( \frac{k}{n} \right)^{s - 1}  c_{\operatorname{max}} \   \frac{m^3}{n},
	\end{equation}
	or equivalently,
	\begin{equation}  \label{eqwww2r2}
	f_1\left( \frac{k}{n} \right) - 
	\left( \frac{k}{n} \right)^{s - 1} c_{\operatorname{max}} \  \frac{m^3}{n}
	\le
	\gamma_{n,k}
	\le
	f_1\left( \frac{k}{n} \right) +
	\left( \frac{k}{n} \right)^{s - 1}  c_{\operatorname{max}} \   \frac{m^3}{n}.
	\end{equation}
	Choose  $n_0$ so large that  $j/n_0 < \delta_0/8$.  Since $k/n \le \delta_0/2$, we have that $(k + j)/n < \delta_0$, so by (\ref{der}), on the interval $[k/n, (k + j)/n]$,
	$f_1^{\prime} (\xi) \ge  s c_s \xi^{s- 1}/2 \ge s c_s (k/n)^{s- 1}/2$.
	By the Mean Value Theorem, 
	
	\begin{equation}
	f_{1}\left( \frac{k+j}{n}\right) -f_{1}\left( \frac{k}{n}\right) 
	\geq 
	\frac{s c_s}{2}
	\left( \frac{k}{n}\right) ^{s-1}
	\left(	\frac{k+j}{n}-\frac{k}{n}\right)\geq \frac{ c_{\operatorname{max}} \  m^3}{n}\left( \frac{k}{n}
	\right) ^{s-1}. 
	\label{eq12b}
	\end{equation}
	From (\ref{eqwww2r2}) and (\ref{eq12b}) we obtain
	\begin{equation}
	\gamma_{n,k}
	\le
	f_{1}\left( \frac{k+j}{n}\right),
	\label{eq123}
	\end{equation} 
	and since $f_1$ is increasing, 
	\begin{equation}
	t_{n,k}
	\le
	\frac{k+j}{n}.
	\label{eq123}
	\end{equation} 
	
	Using the fact that $k \ge 2j$, an entirely analogous argument shows that
	
	\begin{equation*}
	f_{1}\left( \frac{k}{n}\right) -f_{1}\left( \frac{k - j}{n}\right) 
	\geq 
	\frac{s c_s}{2}
	\left( \frac{k - j}{n}\right) ^{s-1}
	\left(\frac{j}{n}\right)
	\geq 
	\frac{s c_s}{2}
	\left( \frac{k}{2 n}\right) ^{s-1}
	\left(\frac{j}{n}\right)\geq \frac{ c_{\operatorname{max}} \  m^3}{n}\left( \frac{k}{n}
	\right) ^{s-1}, 
	\label{eq12}
	\end{equation*}
	so 
	$
	\gamma_{n,k}
	\ge
	f_{1}\left( \frac{k -j}{n}\right)
	$, and
	$
	t_{n,k}
	\ge
	\frac{k -j}{n}.$

	Finally, suppose that  $t_{n,k} < \delta_0/2 \le k/n$. Recalling that $j/n_0 < \delta_0/8$, that $n \ge n_0$, and that  the nodes are nondecreasing, by (\ref{kj}) and (\ref{kgem}) we have
	$$
	\delta_0/4 \le  \delta_0/2 - \frac{1}{n} -  \frac{j}{n} \le \frac{k_0 - 1}{n} - \frac{j}{n}
	\le t_{n, k_0 - 1} \le t_{n,k}.   
	$$
	Thus, the case  $t_{n,k} < \delta_0/2 \le k/n$ has already been considered,
	since then both $k/n, t_{n,k} \in [\delta_0/4,  \delta_0) \subset [\delta_0/4, 1 - \delta_1/4]$.
	\end{proof}

\begin{definition}  The modulus of continuity of a function 
	$f \in C\left[ a,b\right] $ is
	$$\omega\left(f,\delta\right)
	:=\sup\left\{ \left|f(x)-f(y)\right|: x, y \in [a,b],  \left| x-y\right|\leq\delta
	\right\}.$$
\end{definition}

\begin{remark} Consider the classical Bernstein operator $B_n$ over $[0,1]$.
	It is well known that for all $f\in C\left[ 0,1\right] $, all $x\in [0,1]$  and all $n\geq 1$,\begin{equation*}
	\left\vert  f \left( x\right) -
	B_{n}f \left( x\right) \right\vert \leq 
	c \omega \left( f, n^{- \frac{1}{2}}\right),
	\end{equation*}
	where $c = \frac{4306 + 837 \sqrt6}{5832} \approx 1.08988$
	(cf. \cite{Si1}, \cite{Si2}). 
	
	\end{remark}

\begin{theorem} \label{sharp}
	Let $f_{1}:\left[ a,b\right] \rightarrow \mathbb{R}$ be a polynomial
	of degree $m \ge 1$, such that  $f_1^\prime > 0$ on $(a,b)$, 
	and let 
	$
	B_{n}^{f_{1}}$ be the Bernstein operator over $[a,b]$,  fixing $f_{1}$ and the constant
	function $\mathbf{1}$. Denote by  $B_{n}$  the classical Bernstein operator.
	Then
	there exist a constant $K  > 0$ and  a natural number $n_{0}$  such that 
	for all $f\in C\left[ a,b\right] $, all $x\in [a,b]$  and all $n\geq n_{0},$
	\begin{equation*}
	\left\vert  B_{n}^{f_{1}}f \left( x\right) -
	B_{n}f \left( x\right) \right\vert \leq 
	\omega \left( f, K n^{- 1}\right).
	\end{equation*}
\end{theorem}

\begin{proof} 
By Theorem  \ref{K/n}, 
\begin{equation*}
\left| B_{n}^{f_{1}}f  \left( x\right) - B_{n}f  \left(
x\right)\right| 
 =
 \left| \sum_{k=0}^{n}\left( f\left( t_{n,k}\right) -f\left( k/n\right) \right) p_{n,k}\left( x\right) \right|
\end{equation*}
\begin{equation*}
	\le
	 \sum_{k=0}^{n} \left| f\left( t_{n,k}\right) -f\left( k/n\right) \right| p_{n,k}\left( x\right)
	\le
\sum_{k=0}^{n}  \omega \left( f, K n^{- 1}\right) p_{n,k}\left( x\right) 
=
 \omega \left( f, K n^{- 1}\right).
\end{equation*}
\end{proof}

Hence, 	for  $f\in C\left[ 0,1\right] $, $ n \ge K^2$, 
$c = \frac{4306 + 837 \sqrt6}{5832}$ and $x\in [0,1]$, 
\begin{equation*}
\left\vert  f \left( x\right) -
B_{n}^{f_{1}} f \left( x\right) \right\vert 
\le
\left\vert  f \left( x\right) -
	B_{n}f \left( x\right) \right\vert  + \left\vert  B_n f \left( x\right) -
	B_{n}^{f_{1}} f \left( x\right) \right\vert
	\end{equation*}
	\begin{equation*}
	\le
c \omega \left( f, n^{- \frac{1}{2}}\right) + 
\omega \left( f, K n^{-1}\right)
\le
\left(c + 1 \right)  \omega \left( f, n^{- \frac{1}{2}}\right),
\end{equation*}
which is comparable to the rate of convergence  of $B_n$.

\begin{definition} \label{lipopnorm} A function $f:[a,b]\to \mathbb{R}$ is {\em Lipschitz} 
if there exists a constant $K > 0$ such that for all $x, y \in[a,b]$,
$	\left| f(x)-f(y)\right|  
	\le 
	K |x - y|.
$
If for some  $0<\alpha \le 1$ and $K>0$ we have 
$	\left| f(x)-f(y)\right|  
\le 
K |x - y|^{\alpha},
$
then we say that $f$ is {\em H\"older continuous} of order $\alpha$.
\end{definition}

We will often use the shorter expression ``$\alpha$-H\"older". 
Thus, 1-H\"older  is the same as Lipschitz. Furthermore, if $
0 < \alpha < \beta \le 1$ and $f$ is $\beta$-H\"older, then it is
also $\alpha$-H\"older, since
$	\left| f(x)-f(y)\right|  
\le 
K |x - y|^{\beta}
= 
K |x - y|^{\alpha} |x - y|^{\beta - \alpha}
\le 
K  (b - a)^{\beta - \alpha} |x - y|^{\alpha}.
$

The next theorem may be well known, but an internet search has yielded
no results.

\begin{theorem} \label{sHolder} Let $p: [a,b] \to \mathbb{R}$ be an increasing polynomial of degree $m\ge 1$, and let $0 \le s - 1$ be the largest order of the zeros that $p^\prime$ may
	have in $[a,b]$.
	Then there exist a $\delta > 0$ and a $K > 0$ such that 
	for all $u, v \in[p(a) , p(b) ]$ with $ |u - v| \le \delta$,
	$	\left| p^{ -1 }(u)- p^{ -1 }(v)\right|  
	\le  
	K |u - v|^{1/s}.
	$
	Hence, $	\left| p^{ -1 }(u)- p^{ -1 }(v)\right|  
	\le  
	K^\prime  |u - v|^{1/m}
	$
	for some $K^\prime > 0$.
\end{theorem}

\begin{proof} If  $p^{\prime }  >0$ on
	$\left[ a,b\right]$, then $p^{ -1 }$ is Lipschitz on 	$\left[ p(a) , p(b)\right]$ (with constant $K =\|(p^{ -1 })^\prime\|_\infty$) and hence,  H\"older of order $1/s$ = 1. 
	So suppose $p^{\prime }$ vanishes somewhere in  	$[a,b]$, say, at
	the distinct  points $a \le x_0 <  \dots <  x_j \le b$ (if $p^{\prime }$
	has only one zero,  at $x_0$, then $a \le x_0 \le b$, and the preceding notation 
	is
	not intended to imply that  $x_0 < b$).

	For $0 \le i \le j$, let 
	$1 \le s_i - 1 \le m - 1$ be the  order of the zero of $p^\prime$ at $x_i$.
	Then
	\begin{equation}\label{t0}
	p (x) - p(x_i) 
	= 
	\sum_{l=s_i}^{m}c_{l}  (x - x_i)^{l} = c_{s_i}  (x - x_i)^{s_i} + O (|x - x_i|^{s_i+ 1}).
	\end{equation}
	Since $p$ is increasing on $[a,b]$,  for each $i$ there exists   an $\varepsilon_i > 0$ such that if $x_0 = a$, then
	$p$ is convex on $[a, a + \varepsilon_0]$, if $x_i  \in (a,b)$, then 
	$p$ is concave  on $[x_i -\varepsilon_i, x_i]$ and convex on $[x_i, x_i + \varepsilon_i]$, and if
	$x_j = b$, then $p$ is  concave  on $[b -\varepsilon_j, b]$.
	
	Note  that for $x_i \in [a,b)$, we have $c_{s_i} > 0$, since $p$ is increasing
	on $[x_i, b]$ (but if $x_i = b$ and $s_i$ is even, then $c_{s_i} < 0$).
	Suppose first that $x_i \in (a,b)$.
	Select 
	$\delta_i \in (0,  \varepsilon_i]$  so that for all 
	$y \in [x_i, x_i + \delta_i]$, 
	\begin{equation}\label{t1}
	\frac{c_{s_i} }{2}( y - x_i )^{s_i}
	\le 
	p (y) - p(x_i),
	\end{equation}
	with the inequality reversed when 
	$y \in [x_i - \delta_i, x_i]$:
	\begin{equation}\label{t2}
	\frac{c_{s_i} }{2}( y - x_i )^{s_i}
	\ge 
	p (y) - p(x_i).
	\end{equation}
	
	The one sided cases $x_0 = a$ and $x_j = b$ are simpler and handled in the same
	way: with the same notation as above, we also have
	\begin{equation} \label{t3}
	\frac{c_{s_0} }{2}( y - a )^{s_0}
	\le 
	p (y) - p(a)
	\mbox{ \ \ \ and \ \ \ }
	\frac{c_{s_j} }{2}( y - b )^{s_j}
	\ge 
	p (y) - p(b)
	\end{equation}
	on $[a, a + \delta_0]$ and 
	$[b - \delta_j, b]$ respectively.
	Next, choose $\delta >0$ satisfying the following two conditions:

	i) For $i = 0, \dots, j$, 
	$$
	[a,b] \cap p^{-1} ([p(x_i) - 2 \delta, p(x_i) + 2 \delta])
	\subset  
	[a,b] \cap [x_i - \delta_i, x_i + \delta_i].
	$$
	
	ii)  For $i = 0, \dots, j$, the intervals $[p(x_i) - 2 \delta, p(x_i) + 2 \delta]$ are disjoint.
	
	Now let $u,v \in [p(a) , p(b)]$ satisfy $|u - v| \le \delta$.
	On $[p(a), p(b)] \setminus \cup_{i = 0}^j (p(x_i) - \delta, p(x_i) + \delta)$,
	$\|(p^{ -1 })^\prime\|_\infty < \infty$, and hence $p^{ -1 }$ is Lipschitz 
	there. Next, assume that for some
	$i$, either 
	$u\in (p(x_i) - \delta, p(x_i) + \delta)$ or 
	$v\in (p(x_i) - \delta, p(x_i) + \delta)$. 
	Then both $u, v\in (p(x_i) - 2 \delta, p(x_i) + 2 \delta)$.
	By a translation and by substracting a constant if needed, we may, 
	without loss of generality, suppose
	that $0 = x_i = p (x_i)$, so 
	$u, v \in ( - 2 \delta,  2 \delta)$. 
	
	Assume first 
	that $x_i < b$. Since 
	\begin{equation*}
		\frac{c_{s_i} y^{s_i} }{2}
		\le 
		p (y)  \mbox{ \ on \ } [0, \delta_i], \mbox{  \ it follows that  \ }
		\left(\frac{ 2 u }{c_{s_i}}\right)^{1/s_i}
		\ge 
		p^{-1} (u)  \mbox{ \ on \ } [0, 2 \delta).
	\end{equation*}
	By the concavity of  $p^{ -1 }$
	on $ [0,  2 \delta)$,
	if $0 \le u, v$, then 
	$$
	|p^{ -1 } (u) - p^{ -1 }(v)| \le
	p^{ -1 } (|u - v|)
	\le
	\left(\frac{ 2 | u -  v| }{c_{s_i}}\right)^{1/s_i}.
	$$
	If $0 = x_i = a$ we do not need to do anything else. If $0 = x_i \in (a,b)$, 
	and $0 \ge u, v$,  say, with $u < v \le 0$, we argue
	in the same way,  but paying
	attention to the negative signs. The interior zeros of $p^\prime$ have
	even order, since $p^\prime \ge 0$ on $[a,b]$, so $s_i - 1$ is even, and $s_i$, odd.
	By (\ref{t2}), for all $y\in [- \delta_i, 0]$, 
	\begin{equation*}
		\frac{c_{s_i} }{2} y^{s_i} \ge p(y).
	\end{equation*}
	Writing $w = p(y)$ for  $w \in (- 2 \delta,  0]$, we have
	\begin{equation*}
		\frac{c_{s_i} }{2} \left(p^{-1} (w)\right)^{s_i} \ge w,
	\end{equation*}
	or equivalently, 
	$$
	\frac{c_{s_i} }{2} \left|p^{-1} (w)\right|^{s_i} \le |w|;
	$$
	hence,
	\begin{equation} \label{inv}
	|p^{-1} (w)| \le \left(	\frac{2 }{c_{s_i}} |w|\right)^{1/s_i}.	
	\end{equation}
	Now by the convexity of  $p^{ -1 }$
	on $(- 2 \delta,  0]$, 
	$$
	0 > p^{ -1 } (u) - p^{ -1 }(v) \ge
	p^{ -1 } (u - v),
	$$
	so
	\begin{equation} \label{inv2}
	|p^{ -1 } (u) - p^{ -1 }(v)| \le
	|p^{ -1 } (u - v)| 
	\le
	\left(	\frac{2 }{c_{s_i}} \left|u - v\right|\right)^{1/s_i}.
	\end{equation}
	Lastly, if, say $u < 0 < v$, and for instance,
	$v\ge |u|$, then
	$$
	|p^{ -1 } (u) - p^{ -1 }(v)| \le
	|p^{ -1 } (-v) - p^{ -1 }(v)| 
	\le
	2 \left(\frac{ 2 |v| }{c_{s_i}}\right)^{1/s_i}
	\le
	2 \left(\frac{ 2 | u -  v| }{c_{s_i}}\right)^{1/s_i}.
	$$
	To finish, if $0 = x_i = b$ and $s_i$ is odd, the argument
	is exactly as in the case $u, v \in (- 2 \delta,  0]$ seen above, 
	while if $s_i$ is even, then $c_{s_i} < 0$ in (\ref{t0}), since $p$ is increasing
	on $[a,b]$; from (\ref{t3}) we conclude that for $w = p(y)\in (- 2 \delta,  0]$,
	$$
	\frac{|c_{s_i}| }{2} \left(|p^{-1} (w)|\right)^{s_i} \le |w|,
	$$
	and taking $-2 \delta < u  < v \le 0$, we have
	\begin{equation*}
		|p^{ -1 } (u) - p^{ -1 }(v)| \le
		|p^{ -1 } (u - v)| 
		\le
		\left(	\frac{2 }{|c_{s_i}|} |u - v|\right)^{1/s_i}.
	\end{equation*}
	Now  $i$ is arbitrary and $s_i \le s \le m$, so the result follows.
\end{proof}

It is well known and easy to prove, that if $g :[a,b] \to \mathbb{R}$ is  continuous, and there
exist a $\delta > 0$ 
and a $K > 0$ such that 
for all $x, y \in[a,b]$ with $ |x - y| \le \delta$, we have
$	\left| g(x)- g(y)\right|  
\le  
K |x - y|^{\alpha}, 
$
then $g$ is
H\"older continuous of order $\alpha$. 
Though Theorem \ref{sHolder} is sufficient for our purposes, the next result is
interesting in itself.

\begin{corollary} Let $p: [a,b] \to \mathbb{R}$ be an increasing polynomial of degree $m\ge 1$.
	Then $p^{ -1 } : [p(a) , p(b)] \to [a,b]$ is  H\"older continuous of order
	$1/s$, where $s - 1 \ge 0$ denotes the  largest order of any zero of $p^\prime$ 
	in $[a,b]$.
\end{corollary}

\begin{theorem} 
Let $f_{1}$ be a strictly increasing polynomial on $[a,b]$,
	of degree $m$, and let  $s - 1 \ge 0$ denote the largest order of the zeros that $f_1^\prime$ may
	have in
	$\left[a,b\right]$.   Then there exist  constants  $n_0 \ge m$ and  $K > 0$ such that
	$B_n^{f_1}$ 
	is well defined  whenever $n \ge n_0$, and 
	for all  $0\leq k\leq n$,  we have  
	$$
	\left|t_{n,k} - \left(a+\frac{k}{n}\left( b-a\right)\right)\right| 
	\le \frac{K}{n^{1/s}},
	$$	where the points $t_{n, k} = f_{1}^{-1}\left( \gamma_{n,k}\right)$ are the nodes
	 of   $B_n^{f_1}$. 
	\end{theorem}

\begin{proof} The argument to handle the possible zeros of $f_1^\prime$ at the endpoints
is essentially  identical to the one used in the proof of Theorem \ref{K/n}: keeping the analogous choices made there
for $\delta_0$ and $\delta_1$, on
$[a, a +  \delta_0/2)$ and $(b - \delta_1/2 , b]$ we
have that the coordinates $\gamma_{n,k}$ of $f_1$ are non-decreasing
when $k/n \in [a, a +  \delta_0/2) \cup (b - \delta_1/2 , b]$, and since
$f_1$ is increasing, the nodes $t_{n,k}$ are also non-decreasing.

As for the central region  $[a + \delta_0/2 , b - \delta_1/2]$,  
 we use Theorem  \ref{sHolder} for $f_1|_{[a + \delta_0/2 , b - \delta_1/2]}$:
since by Theorem \ref{unifapprox} we have
	$$
	\left|f_1 \left(a+\frac{k}{n}\left( b-a\right)\right)- \gamma_{n,k} \right|  =  O(1/n)
	$$ 
	(with constant depending
	only on $f_1$ and $[a,b]$) and
	$f_1^{-1}$ is $1/s$-H\"older over $f_1 ({[a + \delta_0/2 , b - \delta_1/2]})$  (again with constant depending
	only on $f_1$ and $[a,b]$) 
	we conclude that
		$$
		\left|a+\frac{k}{n}\left( b-a\right)-  t_{n,k} \right| 
		=
\left|f_1^{-1}\left(f_1\left(a+\frac{k}{n}\left( b-a\right)\right)\right) -  f_{1}^{ -1 } (\gamma_{n,k} )\right|
			$$
			$$
			 =
 O\left(\left|f_1\left(a+\frac{k}{n}\left( b-a\right)\right) -  \gamma_{n,k} \right|^{1/s}\right)
			 =
			  O(n^{- 1/s}).
			 $$ 
			 \end{proof}
			 
			 Thus, the following variant of Theorem \ref{sharp}, with 
			 $1/n^{1/s}$ in the modulus of continuity instead of $1/n$,  is obtained.

\begin{theorem} \label{mod}
Let $f_{1}:\left[ a,b\right] \rightarrow \mathbb{R}$ be an increasing polynomial
of degre $m \ge 1$, and let 
$
B_{n}^{f_{1}}$ be the Bernstein operator over $[a,b]$,  fixing $f_{1}$ and the constant
function $\mathbf{1}$. Denote by  $B_{n}$  the classical Bernstein operator,
and denote by   $s - 1 \ge 0$ the largest order of the zeros of  $f_1^\prime$
in $[a,b]$.
Then
there exist a constant $K = K(f_1|_{[a,b]}) > 0$ and  a natural number $n_{0}$  such that 
for all $f\in C\left[ a,b\right] $, all $x\in [a,b]$  and all $n\geq n_{0},$
\begin{equation*}
\left\vert  B_{n}^{f_{1}}f \left( x\right) -
B_{n}f \left( x\right) \right\vert \leq 
 \omega \left( f, K n^{- 1/s}\right).
\end{equation*}
\end{theorem}

\begin{theorem} With the same notation as in the preceding result, for
	all  $f\in C\left[ a,b\right] $,
	$\lim_{n\to \infty} \|B_{n}^{f_{1}} f - f\|_\infty = 0$. 
\end{theorem}

\begin{proof}
Since $f$ is continuous on $[a,b]$, $\lim_{n\to 0} 
 \omega \left( f, K n^{- 1/s} \right) = 0$, so 
by Theorem \ref{mod} and the convergence of the standard Bernstein operator,  
	\begin{equation*}
	\lim_{n\to \infty} \|B_{n}^{f_{1}} f - f\|_\infty
	\le
	\lim_{n\to \infty} \|B_{n}^{f_{1}} f - B_n f\|_\infty
	+
	\lim_{n\to \infty} \|B_{n} f - f\|_\infty = 0.
	\end{equation*}
\end{proof}

We finish with some remarks on shape preservation.
Regarding the convexity preserving properties  of the generalized Bernstein operator $B_n^{f_1}$ when $f_1^\prime > 0$
on $(a,b)$, by  \cite{AKR08b} they are analogous
to the ones of the standard Bernstein operator, but understood with $(\mathbf{1}, f_1)$-convexity 
replacing ordinary convexity.

\begin{definition} Let $E\subset\mathbb{R}$. A function $f: E
	\rightarrow\mathbb{R}$ is called \emph{$(f_{0},f_{1})$-convex } on $E$
	 if for all $x_{0},x_{1},x_{2}$ in $E$ with $
	x_{0}<x_{1}<x_{2} $, the determinant 
	\begin{equation}
	\operatorname{Det}_{x_{0},x_{1},x_{2}}\left( f\right) :=\det\left( 
	\begin{array}{ccc}
	f_{0}\left( x_{0}\right) & f_{0}\left( x_{1}\right) & f_{0}\left(
	x_{2}\right) \\ 
	f_{1}\left( x_{0}\right) & f_{1}\left( x_{1}\right) & f_{1}\left(
	x_{2}\right) \\ 
	f\left( x_{0}\right) & f\left( x_{1}\right) & f\left( x_{2}\right)
	\end{array}
	\right)  \label{defD}
	\end{equation}
	is non-negative. 
\end{definition}
Note in particular that convexity is the same as $(\mathbf{1}, x)$-convexity.
 In \cite[Theorem 25]{AKR08b}
it is shown that when the nodes fail to be non-decreasing, $(\mathbf{1}, f_1)$-convexity
may be lost, but the example given there is not a polynomial space. The same
phenomenon can occur in the polynomial context, as we prove next.

\begin{example} Consider 
	$\mathbb{P}_4[-1,1] =\operatorname{Span}\{1,x, x^2, x^3, x^4\}$,  with the standard
	Bernstein bases over 
	$[-1,1]$. It is easy to check that the coordinates of $f_1(x) := x^3$ are 
	$\gamma_{4,0} = -1$, $\gamma_{4,1}= 1/2$, $\gamma_{4,2} = 0,  \gamma_{4,3} = 
	- 1/2$  and 
	$\gamma_{4,4} = 1$ (just plug in, and simplify; alternatively, these
	coordinates can be obtained from formula (\ref{Berncoor})). 
	The nodes are the cube roots of the corresponding coordinates, so 
	$$
	B^{f_1}_4 e_4(x)
	=
\frac{(1 - x)^4}{16} + \left(\frac{1}{2}\right)^{4/3} \frac{(1 - x)^3 (1 + x)}{4}
+ \left(\frac{1}{2}\right)^{4/3} \frac{(1 - x) (1 + x)^3}{4}
+
\frac{(1 + x)^4}{16}.
	$$ 
	By \cite[Theorem 5]{BePa}, $f$ is $(\mathbf{1}, f_1)$-convex if and only if
$f\circ f_1^{-1}$ is convex in the ordinary sense.
Note that  $e_4 (x) = x^4$ is $(\mathbf{1}, x^3)$-convex,
	since $e_4(x^{1/3}) = (x^{1/3})^4$ is convex. But 
	$
	B^{f_1}_4 e_4(x)
	$ 
	is not
	$(\mathbf{1}, x^3)$-convex: expanding,
	simplifying, and replacing $x$ with $x^{1/3}$, we get 
	$$
	B_4^{f_1} e_4\left(x^{1/3}\right)  = \frac{2^{2/3} + 1}{8} +
	\frac34\left(x^{1/3}\right)^2 -  \left(\frac{2^{2/3} - 1}{8}\right) \left(x^{1/3}\right)^4.
	$$
	Since both $x^{2/3}$ and $ - x^{4/3}$ are concave on $[0,1]$, so
	is $B_4^{f_1} e_4(x^{1/3})$.
\end{example}

\end{document}